\documentclass[10pt, a4paper, twoside]{amsart}
\usepackage{amsthm}
\usepackage{amsmath}
\usepackage{amssymb}
\usepackage{mathtools}
\usepackage[all]{xy}
\usepackage{indentfirst}
\usepackage{comment}

\newtheorem{thm}{\bf Theorem}[section]
\newtheorem{prop}[thm]{\bf Proposition}
\newtheorem{lem}[thm]{\bf Lemma}

\newtheorem*{thm*}{\bf Theorem}
\newtheorem*{cor*}{\bf Corollary}

\theoremstyle{definition}

\newtheorem{rem}[thm]{\it Remark}

\newtheorem*{df*}{\bf Definition}
\newtheorem*{not*}{\bf Notation}
\newtheorem*{ack*}{\bf Acknowledgements}
\newtheorem*{dfrem*}{\bf Definition and Remark}

\newtheorem*{nota*}{\bf Notation}

\newtheorem*{remone*}{\it Remark}
\newtheorem*{remtwo*}{\it Remark}
\newtheorem*{remthree*}{\it Remark}
\newtheorem*{remfour*}{\it Remark}

\def\P{\mathbb{P}}
\def\C{\mathbb{C}}

\def\Q{\mathbb{Q}}

\def\Z{\mathbb{Z}}
\def\R{\mathbb{R}}

\DeclareMathOperator{\Supp}{Supp}
\DeclareMathOperator{\Aut}{Aut}
\DeclareMathOperator{\Bir}{Bir}

\DeclareMathOperator{\Sing}{Sing}

\DeclareMathOperator{\Pic}{Pic}

\DeclareMathOperator{\Sec}{Sec}

\DeclareMathOperator{\Tan}{Tan}

\makeindex

\title[\tiny Birational superrigidity and K-stability]
{Birational superrigidity and K-stability of projectively normal Fano manifolds of index one}
\author{Fumiaki Suzuki}
\address{Department of Mathematics, Statistics, and Computer Science, University of Illinois at Chicago}
\email{fsuzuk2@uic.edu}

\begin{document}
\maketitle

\begin{abstract}
We prove that every projectively normal Fano manifold in $\P^{n+r}$ of index $1$, codimension $r$ and dimension $n\geq 10r$ is birationally superrigid and K-stable.
This result was previously proved by Zhuang under the complete intersection assumption.
\end{abstract}

\section{Introduction} 
Birational superrigidity and K-stability of Fano manifolds are two important notions with different backgrounds. 
The notion of birational superrigidity is motivated by the rationality problem of Fano manifolds. 
A Fano manifold $X$ with the Picard number $1$ is called {\it birationally superrigid}
if any birational map from $X$ to the source of another Mori fiber space is an isomorphism.
It implies that $X$ is non-rational and $\Bir(X)=\Aut(X)$.
On the other hand, the notion of K-stability is motivated by the existence of K\"ahler-Einstein metric on Fano manifolds.
A Fano manifold $X$ is called {\it K-stable} if the Donaldson-Futaki invariant is positive for any non-trivial normal test configuration.
It is stronger than the K-polystability which is equivalent to the existence of K\"ahler-Einstein metric \cite{CDS1, CDS2, CDS3, T}.
Birational superrigidity and K-stability are unexpectedtly related according to Odaka-Okada and Stibitz-Zhuang \cite{OO, SZ},
and it is conjectured by Kim-Okada-Won \cite{KOW} that every birationally superrigid Fano manifold is K-stable.
Both the notions are intensively studied in the case of smooth Fano complete intersections of index $1$:
birational superrigidity by Iskovskih-Manin, Pukhlikov, Cheltsov, de Fernex-Ein-Musta\c t\u a, de Fernex, Suzuki, and Zhuang \cite{IM, P1, P2, P3, P5, P6, C1, dFEM2, dF1, dF2, S, Z} 
(see also the note \cite{K2} written by Koll\'ar),
and K-stability by Fujita, Stibitz-Zhuang, and Zhuang \cite{Fuj, SZ, Z}.
Among them, Zhuang \cite{Z} proves that every smooth Fano complete intersection of index $1$ and small codimension is birationally superrigid and K-stable.

In this paper, we replace the complete intersection assumption of the main theorem of \cite{Z} by the projective normality:

\begin{thm}\label{t1}
Every projectively normal Fano manifold in $\P^{n+r}$ of index $1$, codimension $r$, and dimension $n\geq 10r$
is birationally superrigid and K-stable.
\end{thm}
 
 \begin{rem}
So far we do not have any non-complete-intersection examples which satisfy the assumption of Theorem \ref{t1}.
In fact, the Hartshorne conjecture predicts that every smooth projective variety in $\P^{n+r}$ of codimension $r$ and dimension $n>2r$ is a complete intersection, while the conjecture is widely open.
\end{rem}

A key step involves a generalization of multiplicity bounds for cycles on smooth complete intersections
due to Pukhlikov \cite[Proposition 5]{P4}, Cheltsov \cite[Lemma 13]{C3}, and the author \cite[Proposition 2.1]{S}.

This paper is organized as follows.
In Section $2$, we study the ramification locus of the linear projection from a closed point.
In Section $3$, we prove multiplicity bounds for cycles.
In Section $4$, we prove a stronger version of Theorem \ref{t1}.
In Section $5$, we discuss the singular case.

Throughout this paper, the base field is the field of complex numbers $\C$.

\begin{not*}
Let $X$ be a complete algebraic scheme.
\begin{itemize}
\item We denote by $[X]$ the fundamental cycle of $X$.
\item
For a closed subvariety $Z\subset X$, let $e_{Z}(X)$ be the Samuel multiplicity of $X$ along $Z$.
We extend the definition of Samuel multiplicities to arbitrary cycles by linearlity.
\item
For pure-dimensional cycles $\alpha, \beta$ on $X$ intersecting properly and an irreducible componenet $Z$ of the intersection,
we denote by $i(Z,\alpha\cdot \beta; X)$ the intersection multiplicity of $Z$ in $\alpha \cdot \beta$
whenever the intersection product $\alpha\cdot \beta$ is defined at $Z$.
\item
For a closed subscheme $Z\subset X$, we denote by $s(Z, X)$ be the Segre class of $Z$ in $X$.
\item
For a vector bundle $E$ on $X$, we denote by $c(E)$ the total Chern class of $E$, and by $c_{i}(E)$ the $i$-th Chern class of $E$.
\item 
We denote by $Z_{i}(X)$ the group of $i$-cycles on $X$,
and by $CH_{i}(X)$ (resp. $N_{i}(X)$) the group of those modulo the rational equivalence (resp. the numerical equivalence).
When $X$ is smooth, we denote by $N^{j}(X)$ the group of $j$-cocycles on $X$ modulo the numerical equivalence.
\end{itemize}
For the definitions, we refer the reader to \cite{F}.
\end{not*}

\begin{ack*}
The author wishes to thank his advisor Lawrence Ein for constant support and warm encouragement.
He is grateful to Ziquan Zhuang for helpful comments.
Finally, he thanks the anonymous referee for valuable suggestions which improve the exposition of this paper.
\end{ack*}

\section{The ramification locus of the linear projection from a closed point}
Let $X\subset \P^{n+r}$ be a non-degenerate smooth projective variety of dimension $n$ and codimension $r$.
We take a closed point $p\in \P^{n+r}$ not contained in $X$.
The choice of $p$ determines a section $s\in \Gamma(X,N_{X/\P^{n+r}}(-1))$.
Let $\pi_{p}\colon X\rightarrow \P^{n+r-1}$ be the restriction of the linear projection from $p$.
We define the ramification locus $R(\pi_{p})\subset X$ of $\pi_{p}$ as the zero-scheme of the section $s$
\cite[Example 14.4.15]{F}.
We consider the following condition $(\star)$:
given any closed point $p\in \P^{n+r}$ not contained in $X$,
we have $Z\cap R(\pi_{p})\neq \emptyset$ for any closed subset $Z\subset X$ of dimension $\geq r$.

\begin{lem}\label{l1}
Assume that $X$ is a complete intersection.
Then the condition $(\star)$ is satisfied.
\end{lem}
\begin{proof}
The ramification locus $R(\pi_{p})$ is globally defined by $r$ hypersurfaces on $X$.
The condition $(\star)$ is obviously satisfied.
\end{proof}


\begin{lem}\label{l2}
Assume $n\geq 3r-2$.
Then the condition $(\star)$ is satisfied.
\end{lem}
\begin{proof}
We may assume that $r\geq 2$.
We denote by $\Tan(X)$ the tangent variety of $X$.
We prove $\Tan(X)=\P^{n+r}$.
We denote by $\Sec(X)$ the secant variety of $X$.
We have $\Tan(X)=\Sec(X)$ as long as $n\geq r$
by the connectedness theorem of Fulton and Hansen \cite[Corollary 3.4.5]{L}.
On the other hand, we have $\Sec(X)=\P^{n+r}$ as long as $n\geq 2r$ by Zak's theorem on linear normality \cite[Corollary 3.4.26]{L}.
Therefore we have the desired equality of sets under our assumption.
It follows that $R(\pi_{p})$ is non-empty, and $\dim R(\pi_{p}) = n-r$ if $p$ is general.
We check the condition $(\star)$.
We may assume that $p$ is general so that $R(\pi_{p})$ is defined by a regular section.
We have
\[[R(\pi_{p})]=c_{r}(N_{X/\P^{n+r}}(-1))\cap [X]\]
 by \cite[Proposition 14.1]{F}.
For the right hand side, we have
\[
c_{r}(N_{X/\P^{n+r}}(-1))\cap [X]=\sum_{i\geq0} (-c_{1}(\mathcal{O}_{X}(1)))^{i}\cdot c_{r-i}(N_{X/\P^{n+r}})\cap [X]
\]
by \cite[Example 3.2.2]{F}.
We have
\[
c_{r}(N_{X/\P^{n+r}})\cap [X]= \deg X\cdot c_{1}(\mathcal{O}_{X}(1))^{r}\cap [X]
\]
by the self-intersection formula \cite[Corollary 6.3]{F}.
On the other hand, we have
\[
H^{2i}(X,\Z)=\Z\cdot (c_{1}(\mathcal{O}_{X}(1))^{i}\cap [X]) \text{ for any }0\leq i\leq r-1
\]
by the Barth-Larsen theorem \cite[Theorem 3.2.1]{L}
under the assumption $n\geq 3r-2$.
Therefore we have
\[
[R(\pi_{p})] 
= a\cdot c_{1}(\mathcal{O}_{X}(1))^{r}\cap [X]\in H^{2r}(X,\Z)
\]
for some positive integer $a$.
For any closed subset $Z\subset X$ of dimension $\geq r$, we have
\[
\deg(Z\cdot [R(\pi_{p})]) = a\cdot \deg Z \neq 0,
\]
which implies $Z\cap R(\pi_{p})\neq 0$.
The proof is done.
\end{proof}

\begin{lem}\label{l3}
Assume $r=2$.
Then the condition $(\star)$ is satisfied.
\end{lem}
\begin{proof}
It follows from \cite[Chapter II, Corollary 2.5]{Za}.
\end{proof}

\section{Multiplicity bounds for cycles}
In this section, we prove the following key proposition:
\begin{prop}\label{p1}
Let $X\subset \P^{N}$ be a non-degenerate smooth projective variety of codimension $r$.
Assume that 
$X$ satisfies 
the condition $(\star)$.
Let $\alpha$ be an effective cycle on $X$ of codimension $s$ such that 
\[\alpha = m \cdot c_{1}(\mathcal{O}_{X}(1))^{s}\cap [X] \in N^{s}(X).\]
Then we have $e_{x}(\alpha)\leq m$ for any closed point $x\in X$ away from a closed subset of dimension $<rs$,
where we use the convention $\dim(\emptyset)=-1$.
\end{prop}
\begin{rem}
The complete intersection case is proved by Pukhlikov \cite[Proposition 5]{P4}, Cheltsov \cite[Lemma 13]{C3}, and the author \cite[Proposition 2.1]{S}.
\end{rem}

\noindent
{\it Step} $1$:
We review the construction of residual intersection classes due to Fulton.

\begin{thm}[\cite{F}, Theorem 9.2]\label{t}
We consider a diagram
\[\xymatrix{
& R\ar[d]^{b} &\\
D\ar[r]^{a} & W \ar[r]^{j} \ar[d]^{g} & V \ar[d]^{f}\\
& X\ar[r]^{i} & Y
},
\]
where the square is a fiber square, $i, j, a, b$ are closed embeddings, and $V$ is a $k$-dimensional variety.
Assume that:
\begin{enumerate}
\item[(i)] $i$ is a regular embedding of codimension $r$;
\item[(ii)] $ja$ embeds $D$ as a Cartier divisor of $V$;
\item[(iii)] $R$ is the residual scheme to $D$ in $W$.
\end{enumerate}
Let $N=g^{*}N_{X}Y$ and $\mathcal{O}(-D) = j^{*}\mathcal{O}_{V}(-D)$.
We define the residual intersection class $\R\in CH_{k-r}(R)$ by the formula
\[
\R=\left\{c(N\otimes \mathcal{O}(-D))\cap s(R, V)\right\}_{k-r}.
\]
Then we have
\begin{eqnarray*}
X\cdot_{Y} V = \left\{c(N)\cap s(D, V)\right\}_{k-r} + \R
\end{eqnarray*}
in $CH_{k-r}(W)$.
\end{thm}

\begin{rem}
In the setting of Theorem \ref{t}, the class $\R$ is represented by an effective cycle if $\dim R =k-r$.
Indeed, let $R_{1},\cdots, R_{t}$ be the irreducible components of $R$.
Then we have
\[
\R = s(R, V)_{k-r} = \sum_{i=1}^{t} e_{i}[R_{i}],
\]
where $e_{i}$ is the same as Samuel's multiplicity $e(q)$ of the primary ideal $q$ determined by $R$ in the local ring $\mathcal{O}_{V, R_{i}}$, which is positive \cite[Corollary 9.2.2]{F}.
\end{rem}

We construct a residual intersection class associated to the linear projection from a closed point.
Let $X\subsetneq \P^{N}$ be a non-degenerate smooth projective variety of codimension $r$.
We assume that $X$ satisfies the condition $(\star)$.
Let $Z\subsetneq X$ be a closed subvariety with $\dim Z\geq r$.
Let $p\in \P^{N}$ be a general closed point.
Let $C=\overline{\bigcup_{x\in Z} \langle p,x \rangle}$ be the cone of $Z$ with the vertex $p$.
Let $R$ be the residual set to $Z$ in $X\cap C$.
We have a diagram with a fiber square
\begin{eqnarray}\label{d1}
\xymatrix{
& R\ar[d] &\\
Z\ar[r] & X\cap C \ar[r] \ar[d] & C \ar[d]\\
& X\ar[r] & \P^{N}
}.
\end{eqnarray}
We want to
define the residual intersection class $\R\in CH_{\dim Z+1-1}(R)$, 
but 
$Z\subset C$ is not a Cartier divisor in general.

To remedy this situation, 
we consider the projective bundle $\P=\P_{\P^{N}}(\mathcal{O}\oplus \mathcal{O}(-1))$ with closed subvarieties $E=\P_{\P^{N}}(\mathcal{O})$, $F=\P_{\P^{N}}(\mathcal{O}(-1))$, and $\widehat{X}=\P_{X}(\mathcal{O}\oplus \mathcal{O}(-1))$.
The projective bundle $\P$ is the blow-up of $\P^{N+1}$ along a closed point $q$ with the exceptional divisor $E$;
the natural projection $\P\rightarrow \P^{N}$ is the resolution of the linear projection $\P^{N+1}\dashrightarrow \P^{N}$ from $q$.
Moreover we identify $F$ with the base $\P^{N}$, and also with its image in $\P^{N+1}$, a hyperplane disjoint from $q$.
Let $\widetilde{p}\in \langle p, q\rangle$ be a closed point different from $p\in F$ and $q$.
Let $\widetilde{C}$ be the cone of $Z\subset F$ with the vertex $\widetilde{p}$.
Then $Z\subset \widetilde{C}$ is a Cartier divisor.
In addition, we have $q\not\in\widetilde{C}$ by the generality of $p$, thus we identify $\widetilde{C}$ with its inverse image in $\P$.
We let $\widetilde{R}$ be the residual scheme to $Z$ in $\widehat{X}\cap \widetilde{C}$.
We have a diagram with a fiber square
\begin{eqnarray}\label{d2}
\xymatrix{
& \widetilde{R}\ar[d] &\\
Z\ar[r] & \widehat{X}\cap \widetilde{C} \ar[r] \ar[d] & \widetilde{C} \ar[d]\\
& X \ar[r] & \P^{N}
}.
\end{eqnarray}
Let $\widetilde{\R}\in CH_{\dim Z+1 -r}(\widetilde{R})$ be the residual intersection class.
Let $f\colon \widetilde{C}\rightarrow C$ be the restriction of the natural projection $\P\rightarrow \P^{N}$, which is a finite morphism.
We replace $R$ by $f(\widetilde{R})$ if necessary.
We define $\R$ to be $f_{*}\widetilde{\R}$.

\begin{lem}\label{l7}
We have
\[c_{1}(\mathcal{O}_{X}(1))\cap \R = c_{r}(N_{X/\P^{N}}(-1))\cap [Z].\]
\end{lem}

\begin{proof}
We apply Theorem \ref{t} to the diagram (\ref{d2}).
We have
\[
X\cdot_{\P^{N}} \widetilde{C} = \left\{c(N)\cap s(Z, \widetilde{C})\right\}_{\dim Z+1-r} + \widetilde{\R}
\]
in $CH_{\dim Z +1 -r}(\widehat{X}\cap \widetilde{C})$.
Since the morphism $f$ is generically one-to-one by the generality of $p$, we have $f_{*}[\widetilde{C}]=[C]$.
Thus we have
\[
X\cdot C =f_{*}(X\cdot_{\P^{N}} \widetilde{C})= \left\{c(N)\cap f_{*}s(Z, \widetilde{C})\right\}_{\dim Z +1-r} + \R
\]
in $CH_{\dim Z+1-r}(X\cap C)$,
where the first equality follows from \cite[Theorem 6.2 (a)]{F}.
Since the 
universal sub line bundles 
on $\P^{N}$, $\P^{N+1}$, and 
$\P$
all coincide on $\widetilde{C}$,
we use the same notation $\mathcal{O}(-1)$ when restricted.
We have
\begin{eqnarray*}
c(N)\cap f_{*}s(Z, \widetilde{C}) &=& c(N)\cap f_{*}\left(c(\mathcal{O}(1))^{-1}\cap [Z]\right)\\
&=& \sum_{i,j\geq 0} (-c_{1}(\mathcal{O}(1)))^{i}\cdot c_{j}(N)\cap [Z],
\end{eqnarray*}
where the first equality follows from \cite[Proposition 4.1 (a)]{F}.
Then we have
\[
\R  = X\cdot C + \sum_{i\geq 1}(-1)^{i}c_{1}(\mathcal{O}_{X}(1))^{i-1}\cdot c_{r-i}(N_{X/\P^{N}})\cap [Z].
\]
Applying
$c_{1}(\mathcal{O}_{X}(1))\cap$
to the both sides,
we have
\begin{eqnarray*}
 c_{1}(\mathcal{O}_{X}(1))\cap \R &=& X\cdot Z + \sum_{i\geq 1}(-c_{1}(\mathcal{O}_{X}(1)))^{i}\cdot c_{r-i}(N_{X/\P^{N}})\cap [Z]\\
 &=& c_{r}(N_{X/\P^{N}}(-1))\cap [Z],
\end{eqnarray*}
where the first equality follows from \cite[Proposition 6.3]{F} with
\[
c_{1}(\mathcal{O}(1))\cap [C] = f_{*}(c_{1}(\mathcal{O}(1))\cap [\widetilde{C}])=[Z]
\]
and the second equality follows from the self-intersection formula \cite[Corollary 6.3]{F}.
The proof is done.
\end{proof}

\begin{lem}
As a set, we have
\[
Z\cap \widetilde{R}= Z\cap R(\pi_{p}).
\]
\end{lem}
\begin{proof}
It follows from the same argument as in the proof of \cite[Lemma 3]{P4}.
\end{proof}

By the condition $(\star)$, we have $Z\cap R(\pi_{p})\neq \emptyset$ and $\dim Z\cap R(\pi_{p}) = \dim Z -r$.
Therefore we have $Z\cap \widetilde{R}\neq \emptyset$ and $\dim Z\cap \widetilde{R} = \dim Z -r$.
It implies that $\dim \widetilde{R} = \dim Z+1-r$.
Therefore the class $\widetilde{\R}$ is represented by an effective ($\dim Z +1-r$)-cycle whose support is $\widetilde{R}$,
which implies that the class $\R$ is represented by an effective ($\dim Z+1-r$)-cycle whose support is $R$.

\begin{rem}
There is another possible definition of the residual intersection class $\R'\in CH_{\dim Z+1-r}(R)$ for the diagram (\ref{d1}) using the blow-up along $Z$.
We refer the reader to \cite[Definition 9.2.2]{F} for the details.
We prove $\R'=\R$.
We apply \cite[Corollary 9.2.3]{F} to the diagram (\ref{d1}).
We have
\[
X\cdot C = \left\{c(N)\cap s(Z, C)\right\}_{\dim Z+1-r} +\R'.
\]
The inverse image scheme $f^{-1}(Z)\subset \widetilde{C}$
is the union of $Z$ and a closed subscheme supported on the inverse image $f^{-1}(Z^{ne})$ of 
the non-embedding locus $Z^{ne}$ on $Z$ of the linear projection from $p$.
We have
\[
\left\{c(N)\cap s(Z, C)\right\}_{\dim Z+1-r} =\left\{c(N)\cap f_{*}s(Z, \widetilde{C})\right\}_{\dim Z+1-r}
\]
by \cite[Proposition 4.2]{F} together with $\dim Z^{ne}\leq \dim Z-r$.
Therefore we have the desired equality.
The proof is done.
\end{rem}

\noindent
{\it Step} $2$:
We move cycles by multiple residual intersections.
\begin{lem}\label{mcbmri}
Let $X\subset \P^{N}$ be a non-degenerate smooth projective variety of codimension $r$.
Assume that $X$ satisfies the condition ($\star$).
Let $A\subset X$ be a closed subset of codimension $s$.
Let $Z\subset A$ be a subvariety of dimension $rs$.
Then there is a cycle $\beta$ on $X$ such that
\begin{enumerate}
\item[(i)] $\dim \beta =s$, and $\beta$ intersects $A$ in finitely many points;
\item[(ii)] $|Z\cap \Supp(\beta)|\geq \deg \beta$.
\end{enumerate}
\end{lem}
\begin{proof}
The proof is essentially the same as in \cite[Proposition 5]{P4} (see also \cite[Proposition 2.1]{S}).
We only give a sketch.
We take a general closed point $p\in \P^{N}$ and construct $R$ and $\R$ as in Step $1$.
We define $p_{1}=p$, $R_{1}=R$ and $\R_{1}=\R$.
We replace $R_{0}=Z$ by $R_{1}$, and repeat the same procedure to define $p_{j}$, $R_{j}$ and $\R_{j}$ for $j=1, \cdots, s$.
We have $\dim R_{j}=rs-(r-1)j$
and
\[
\deg \R_{j} = \deg\left(\left(c_{r}(N_{X/\P^{N}}(-1)\right)^{j}\cap [Z]\right).
\]
In particular, we have $\dim R_{s}= s$
and
\[
\deg \R_{s} = \deg\left(\left(c_{r}(N_{X/\P^{N}}(-1)) \right)^{s}\cap [Z]\right).
\]
By careful dimension count using the ramification locus and joins of varieties as in \cite[Lemma 1]{P4} (see also \cite[Lemma 2.3]{S} and \cite[Lemma 2.4]{S}) and by the condition $(\star)$,
we have
\[
\dim A\cap R_{j}=\dim A\cap R_{j-1}\cap R(\pi_{p})=\dim A\cap R_{j-1}-r.
\]
By induction, we have
\[
\dim A\cap R_{j} 
=rs-rj.
\]
In particular, we have
\[
\dim A\cap R_{s}=0.
\]
Moreover we have
\[
Z\cap R_{s}\supseteq Z\cap R_{1}\cap \cdots \cap R_{s} \supseteq Z\cap R(\pi_{p_{1}})\cap \cdots \cap R(\pi_{p_{s}})
\]
and
\begin{eqnarray*}
|Z\cap R(\pi_{p_{1}})\cap \cdots \cap R(\pi_{p_{s}})|&=&Z\cdot [R(\pi_{p_{1}}))]\cdot \ldots \cdot [R(\pi_{p_{s}})]\\
 &=& \deg\left(\left(c_{r}(N_{X}/\P^{N}(-1)) \right)^{s}\cap [Z]\right).
\end{eqnarray*}
Therefore $\beta = \R_{s}$ satisfies the condition.
\end{proof}

\noindent
{\it Step} $3$:

\begin{proof}[Proof of Proposition \ref{p1}]
We take $\alpha$ as in the statement.
By the upper semicontinuity of Samuel multiplicity,
it is enough to prove that $e_{Z}(\alpha)\leq m$ for any closed subvariety $Z\subset X$ of dimension $rs$.
We may assume that $\alpha \neq 0$ and $Z\subset \Supp(\alpha)$.
Then there is an effective cycle $\beta$ on $X$ such that
\begin{enumerate}
\item[(i)] $\dim \beta = s$, and $\beta$ intersects $\alpha$ in finitely many closed points;
\item[(ii)] $|Z\cap \Supp(\beta)|\geq \deg \beta$.
\end{enumerate}
We have
\begin{eqnarray*}
m\cdot\deg(\beta) &=&\alpha\cdot \beta \\
&=& \sum_{x\in \Supp(\alpha)\cap\Supp(\beta)}i(x, \alpha\cdot\beta; X)\\
&\geq& \sum_{x\in Z\cap \Supp(\beta)}i(x, \alpha\cdot\beta; X)\\
&\geq& \sum_{x\in Z\cap \Supp(\beta)}e_{x}(\alpha)\cdot e_{x}(\beta)\\
&\geq& e_{Z}(\alpha)\cdot \deg(\beta),
\end{eqnarray*}
where the second (resp. third) inequality follows from \cite[Corollary 12.4]{F} (resp. the upper semi-continuity of Samuel multiplicity).
We divide the both sides by $\deg(\beta)$.
The proof is done.
\end{proof}

\section{Proof of Theorem \ref{t1}}

We prove a stronger version of Theorem \ref{t1}:
\begin{thm}\label{t1'}
Let $X\subset \P^{n+r}$ be a Fano manifold of index $1$, dimension $n$ and codimension $r$.
Assume that $X$ is $2r$-normal, that is, the restriction map 
\[H^{0}(\P^{n+r},\mathcal{O}_{\P^{n+r}}(2r))\rightarrow H^{0}(X,\mathcal{O}_{X}(2r))\]
 is surjective, and $n\geq 10 r$. 
Then $X$ is birationally superrigid and K-stable.
\end{thm}

\begin{lem}\label{l}
Let $X\subset \P^{N}$ be a Fano manifold with the anti-canonical class a multiple of the hyperplane section class. 
Let $Y\subset X$ be a positive-dimensional linear section.
If $X$ is $k$-normal, so is $Y$.
\end{lem}
\begin{proof}
It is enough to prove that, if $Y\subset X$ is a positive-dimensional linear section, the restriction map
\[
H^{0}(X,\mathcal{O}_{X}(k))\rightarrow H^{0}(Y,\mathcal{O}_{Y}(k))
\]
is surjective for any $k\in \Z$.
It is enough to prove that, if $Y\subset X$ is a linear section, we have
\[
H^{i}(Y,\mathcal{O}_{Y}(j))=0 \text{ for any }0<i<\dim Y\text{ and  }j\in \Z.
\]
It is enough to prove
\[
H^{i}(X, \mathcal{O}_{X}(j))=0 \text{ for any }0<i<n \text{ and }j\in\Z.
\]
It follows from the Kodaira vanishing theorem.
The proof is done.
\end{proof}

\begin{proof}[Proof of Theorem \ref{t1'}]
The proof is essentially the same as in \cite[Theorem 1.2]{Z} and \cite[Lemma 3.6]{Z}.
Let $X$ as in the statement.
We have $-K_{X}=c_{1}(\mathcal{O}_{X}(1))\cap [X]$ and $r\geq 1$.
We have 
$\Pic(X)=\Z[\mathcal{O}_{X}(1)]$
as long as $n\geq r+2$ by the Barth-Larsen theorem \cite[Theorem 3.2.1]{L}.
The inequality $n\geq r+2$ follows from $r\geq 1$ and the assumption $n\geq 10r$.
We may assume that $X$ is non-degenerate.
The condition $(\star)$ is satisfied by Lemma \ref{l2} together with the assumption $n\geq 10r$
(or by Lemma \ref{l3} in the case $r=2$).

We prove that $X$ is birationally superrigid.
By the Noether-Fano enequality \cite[Theorem 1.26]{CS},
it is enough to prove that for any movable linear system $\mathcal{M}\subset |-m K_{X}|$, the pair $\left(X, \frac{1}{m}\mathcal{M}\right)$ is canonical.

\noindent
{\it Step} $1$:
We prove that 
\begin{enumerate}
\item[(i)] there exists a closed subset of dimension $Z\subset X$ of dimension $\leq r-1$ 
such that the pair $\left(X, \frac{1}{m}\mathcal{M} \right)$ is canonical away from $Z$;
\item[(ii)] there exists a closed subset $Z'\subset X$ of dimension $\leq 2r-1$
such that the pair $\left(X,\frac{2}{m}\mathcal{M}\right)$ is log canonical away from $Z'$.
\end{enumerate}
For $(i)$, we take $D\in \mathcal{M}$.
Then 
\[
[D]=m\cdot c_{1}(\mathcal{O}_{X}(1))\cap [X]\in N^{1}(X).
\]
By Proposition \ref{p1}, there exists a closed subset $Z$ of dimension $\leq r-1$ such that $e_{x}(D)\leq m$ for any closed point $x\in X$ away from $Z$.
Therefore the pair $\left(X, \frac{1}{m}D \right)$ is canonical away from $Z$ by \cite[3.14.1]{K1}.
For (ii), we take $D_{1}, D_{2} \in \mathcal{M}$ intersecting properly.
Let 
$B=D_{1}\cap D_{2}$.
Then 
\[
[B]=m^{2}\cdot c_{1}(\mathcal{O}_{X}(1))^{2}\cap [X]\in N^{2}(X).
\]
By Proposition \ref{p1}, there exists a closed subset $Z'\subset X$ such that $e_{x}(B)\leq m^{2}$ for any closed point $x\in X$ away from $Z'$.
For any such $x$, let $S\subset X$ be a general linear section of dimension $2$ through $x$.
Then the pair $\left(S, \frac{2}{m}B|_{S}\right)$ is log canonical at $x$ by \cite[Theorem 0.1]{dFEM1}.
By inversion of adjunction \cite[Theorem 1.1]{EM},
the pair $\left(X, \frac{2}{m}B\right)$ is log canonical at $x$.

\noindent
{\it Step} $2$:
We prove that for any closed point $x\in X$ and any general linear section $Y\subset X$ of codimension $2r-1$ through $x$,
the pair $\left(Y, \frac{1}{m}\mathcal{M}|_{Y}\right)$ is Kawamata log terminal (klt) at $x$.
We have $K_{Y}=2(r-1)\cdot c_{1}(\mathcal{O}_{Y}(1))\cap[Y]$.
Let $L=\mathcal{O}_{X}(2r)$.
Then $L\sim_{\Q}K_{Y}+\frac{2}{m}\mathcal{M}|_{Y}$.
The pair $\left(Y, (1-\epsilon)\frac{2}{m}\mathcal{M}|_{Y}\right)$ is klt away from a finite set for all $0<\epsilon\ll 1$,
and we have
\[
h^{0}(Y, L)\leq h^{0}(\P^{n-r+1}, \mathcal{O}_{\P^{n-r+1}}(2r))=\binom{n+r+1}{2r}<\frac{(n-2r+1)^{n-2r+1}}{(n-2r+1)!},
\]
where the first (resp. second) inequality follows from Lemma \ref{l} together with the $2r$-normality of $X$ (resp. the assumption $n\geq 10r$).
Then the pair $\left(Y, \frac{1}{m}\mathcal{M}|_{Y}\right)$ is klt by \cite[Corollary 1.8]{Z}.

\noindent
{\it Step} $3$:
Assume that the pair $\left(X, \frac{1}{m}\mathcal{M} \right)$ is not canonical at $x\in Z$.
Let $Y\subset X$ be a general linear section of codimension $2r-1$ throught $x$.
Then the pair $\left(Y, \frac{1}{m}\mathcal{M}|_{Y}\right)$ is not log canonical at $x$ by inversion of adjunction, a contradiction.

We prove that $X$ is $K$-stable.
Let 
\[
\alpha(X)
=\sup\left\{t \mid \left(X, t D\right) \text{ is log canonical for any }D\in |-K_{X}|_{\Q}\right\}
\]
be the alpha invariant.
By \cite[Theorem 1.2]{SZ}, it is enough to prove that $\alpha(X)>\frac{1}{2}$.
By \cite[Theorem 1.5]{B}, it is enough to prove that the pair $(X, \frac{1}{2}D)$ is klt for any $D\in|-K_{X}|_{\Q}$.
The proof is similar using Proposition \ref{p1} and \cite[Corollary 1.8]{Z}.

The proof is done.
\end{proof}

By analyzing the proof of Theorem \ref{t1'}, we can further strengthen the statement:

\begin{thm}\label{t1''}
Let $X\subset \P^{n+r}$ be a Fano manifold of index $1$, dimension $n$ and codimension $r$.
Assume
\[\sum_{i=0}^{2r-1}(-1)^{i}\cdot \binom{2r-1}{i} \cdot h^{0}(\mathcal{O}_{X}(2r-i)) < \frac{(n-2r+1)^{n-2r+1}}{(n-2r+1)!},\]
and $n\geq 3r-2$.
Then $X$ is birationally superrigid and K-stable.
\end{thm}
\begin{proof}
We have
\[
h^{0}(Y,\mathcal{O}_{Y}(2r))=\sum_{i=0}^{2r-1}(-1)^{i}\cdot \binom{2r-1}{i} \cdot h^{0}(\mathcal{O}_{X}(2r-i))
\]
for any linear section $Y\subset X$ of codimension $2r-1$
by an argument similar to one in Lemma \ref{l}.
Therefore the first assumption is equivalent to
\[
h^{0}(Y,\mathcal{O}_{Y}(2r))<\frac{(n-2r+1)^{n-2r+1}}{(n-2r+1)!}
\]
for any such $Y$.
We have $n>2r$ after all.
We omit the rest of the proof.
\end{proof}

\section{The singular case}
Due to Liu and Zhuang \cite{LZ}, the result of Zhuang \cite{Z} is generalized to the singular case
(the notions of birational superrigidity and K-stability can be defined for $\Q$-Fano varieties).
Replacing the complete intersection assumption by the local complete intersection and projective normality, we prove:

\begin{thm}\label{t2}
For integers $\delta\geq -1$ and $r\geq 1$,
there exists a positive integer $n_{0}(r, \delta)$ depending only on $\delta$ and $r$ such that,
if $X\subset \P^{n+r}$ is a locally complete intersection projectively normal Fano variety of index $1$, codimension $r$ and dimension $n\geq n_{0}(r, \delta)$ such that
\begin{enumerate}
\item[(i)] $\dim\Sing(X)\leq \delta$;
\item[(ii)] every projective tangent cone of $X$ is a Fano complete intersection of index at least $4r+2\delta +2$ and is smooth in dimension $r+\delta$,
\end{enumerate}
then $X$ is birationally superrigid and K-stable.
\end{thm}

Let $X\subset \P^{n+r}$ be a non-degenerate projective variety of dimension $n$ and codimension $r$.
For a closed point $p\in \P^{n+r}$ not contained in $X$,
we define the ramification locus $R(\pi_{p})$ of the restriction $\pi_{p}\colon X\rightarrow \P^{n+r-1}$  of the linear projection from $p$
as the zero scheme of the section of the twisted normal sheaf $\mathcal{N}_{X/\P^{n+r}}(-1)$ associated to $p$.
We consider the following condition $(\star\star)$:
given any closed point $p\in \P^{n+r}$ not contained in $X$,
we have $Z\cap R(\pi_{p})\neq \emptyset$ for any closed subset $Z\subset X$ of dimension $\geq r$ disjoint from $\Sing(X)$.

\begin{lem}\label{l4}
Assume that $X$ is a complete intersection. 
Then the condition $(\star\star)$ is satisfied.
\end{lem}
\begin{proof}
The proof is the same as in Lemma \ref{l1}.
\end{proof}


\begin{lem}\label{l5}
Assume that $X$ is locally complete intersection and 
\[n\geq \max\left\{3r-2, 2r-1+\delta\right\},\]
 where $\delta=\dim \Sing(X)$.
Then the condition $(\star\star)$ is satisfied.
\end{lem}
\begin{proof}
We denote by $\Tan'(X)$ the variety of tangent stars of $X$ (see \cite[Chapter I, Definition 1.2]{Za} for the definition).
We prove $\Tan'(X)=\P^{n+r}$.
We have $\Tan'(X)=\Sec(X)$ as long as $n\geq r$ by the connectedness theorem of Fulton and Hansen \cite[Chapter I, Theorem 1.4]{Za}.
On the other hand, we have $\Sec(X)=\P^{n+r}$ as long as $n\geq 2r-1+\delta$ by Zak's theorem on linear normality for singular varieties \cite[Chapter II, Theorem 2.1]{Za}.
Therefore we have the desired equality of sets under our assumption.

For a closed point $p\in \P^{n+r}$ not contained in $X$, we denote by $R'(\pi_{p})$ be the J-ramification locus of the restriction of the linear projection from $p$
(see \cite[Chapter II, Section 1]{Za} for the definition and property of unramified morphisms in the sense of Johnson, or J-unramified morphisms).
Then $R'(\pi_{p})\neq \emptyset$, and $\dim R'(\pi_{p})\leq n-r$ for general $p$.
We prove that $R(\pi_{p})\neq \emptyset$, and $\dim R(\pi_{p})=n-r$ for general $p$.
By definition of $R(\pi_{p})$ and $R'(\pi_{p})$, we have
\[
R(\pi_{p})\supseteq R'(\pi_{p}),\, R(\pi_{p})\cap X^{sm}=R'(\pi_{p})\cap X^{sm}.
\]
Thus the complement in $R(\pi_{p})$ of $R'(\pi_{p})$ is supported on $\Sing(X)$,
while we have $\dim \Sing(X)=\delta\leq n-r$ by the assumption.
Now it is enough to observe that $R(\pi_{p})$ is locally defined by $r$ equations in $X$ and we have $\dim R(\pi_{p})\geq n-r$.
It follows that $R(\pi_{p})$ is defined by a regular section for general $p$.
We have
\[
[R(\pi_{p})]=c_{r}(N_{X/\P^{n+r}}(-1))\cap [X]
\]
by \cite[Proposition 14.1]{F}.

To check the condition $(\star\star)$, it is enough to prove that
\[
c_{r}(N_{X/\P^{n+r}}(-1))\cap [Z]\neq 0
\]
for any closed subvariety $Z\subset X$ of dimension $r$.
By the Barth-Larsen theorem for locally complete intersection varieties \cite[Corollary 3.5.13]{L}, 
we have
\[
H_{2i}(X,\Z)=H_{2i}(\P^{n+r}, \Z) \text{ for any }0\leq i\leq r-1.
\]
Therefore the homology class of a subvariety of $X$ of dimension $\leq r-1$ is uniquely determined by its degree.
Combined with the self-intersection formula \cite[Corollary 6.3]{F},
it follows that the function
\[
Z_{r}(X)\rightarrow \Z,\, Z \mapsto c_{r}(N_{X/\P^{n+r}}(-1))\cap [Z]
\]
is $a\cdot \deg Z$, where $a$ is a constant not depending on $Z$.
Taking a general linear section of dimension $r$, we have $a\neq 0$.
The proof is done.
\end{proof}

\begin{lem}\label{l6}
Assume $r=2$.
Then the condition $(\star\star)$ is satisfied.
\end{lem}
\begin{proof}
It follows from \cite[Chapter II, Corollary 2.5]{Za}.
\end{proof}

\begin{prop}\label{p2}
Let $X\subset \P^{N}$ be a non-degenerate projective variety of codimension $r$.
Assume that $X$ satisfies the condition $(\star\star)$.
Let $\alpha$ be an effective cycle on $X$ such that 
\[\alpha= m \cdot c_{1}(\mathcal{O}_{X}(1))^{s}\cap [X] \in N_{N-r-s}(X).\]
Then we have $e_{x}(\alpha)\leq m$ for any closed point $x\in X$ away from a closed subset of dimension $\leq rs+\delta$,
where $\delta = \dim \Sing(X)$ and we use the convention $\dim(\emptyset)=-1$.
\end{prop}
\begin{rem}
The complete intersection case is proved by Pukhlikov \cite[Proposition 5]{P4} and the author \cite[Proposition 2.1]{S}.
\end{rem}
\begin{proof}
We take $\alpha$ as in the statement.
It is enough to prove that $e_{Z}(\alpha)\leq m$ for any closed subvariety $Z\subset X$ of dimension $rs$ disjoint from $\Sing(X)$.
We move $Z$ by multiple residual intersections so that the residual intersections avoid $\Sing(X)$.
The rest of the proof is similar.
\end{proof}

\begin{proof}[Proof of Theorem \ref{t2}]
The proof is essentially the same as in \cite[Theorem 1.3]{LZ}.
We use the Barth-Larsen theorem for locally complete intersection varieties \cite[Corollary 3.5.13]{L} combined with the universal coefficient theorem to compute the Picard group.
The condition $(\star\star)$ is satisfied by Lemma \ref{l5} when $n$ is large enough for fixed $\delta$ and $r$ (or by Lemma \ref{l6} in the case $r=2$).
Then Proposition \ref{p2} is used to bound the dimension of the non-canonical, non-klt or non-lc locus of certain pairs on $X$
as Proposition \ref{p1} in the proof of Theorem \ref{t1}.
We refer the reader to \cite{LZ} for the details.
\end{proof}


\begin{thebibliography}{99}
\bibitem{B}Birkar, C.: {\it Singularities of linear systems and boundedness of Fano varieties}, 2016. arXiv:1609.05543
\bibitem{C1}Cheltsov, I. A.: {\it On a smooth four-dimensional quintic}, Mat. Sb. {\bf 191} (2000), no. 9, 139--160; translation in Sb. Math. {\bf 191} (2000), no.~9-10, 1399--1419. 
\bibitem{C2}Cheltsov, I. A.: {\it Nonrationality of a four-dimensional smooth complete intersection of a quadric and a quadric not containing a plane}, (Russian) Mat. Sb. {\bf 194} (2003), no. 11, 95--116; translation in Sb. Math. {\bf 194} (2003), no.~11-12, 1679--1699.
\bibitem{C3}Cheltsov, I. A.: {\it Nonexistence of elliptic structures on general Fano complete intersections of index one}, Vestnik Moskov. Univ. Ser. I Mat. Mekh. {\bf 2005}, no. 3, 50--53, 71; translation in Moscow Univ. Math. Bull. {\bf 60} (2005), no.~3, 30--33 (2006).
\bibitem{CS}Cheltsov, I. A., Shramov, K. A. :{\it Log-canonical thresholds for smooth Fano threefolds}, Russian Math. Surveys {\bf 63} (2008), no.~5, 859--958; translated from Uspekhi Mat. Nauk {\bf 63} (2008), no. 5(383), 73--180.
\bibitem{CDS1}Chen, X., Donaldson, S., Sun, S.: {\it K\"{a}hler-Einstein metrics on Fano manifolds. I: Approximation of metrics with cone singularities}, J. Amer. Math. Soc. {\bf 28} (2015), no.~1, 183--197.
\bibitem{CDS2}Chen, X., Donaldson, S., Sun, S.: {\it K\"{a}hler-Einstein metrics on Fano manifolds. II: Limits with cone angle less than $2\pi$}, J. Amer. Math. Soc. {\bf 28} (2015), no.~1, 199--234.
\bibitem{CDS3}Chen, X., Donaldson, S., Sun, S.: {\it K\"{a}hler-Einstein metrics on Fano manifolds. III: Limits as cone angle approaches $2\pi$ and completion of the main proof}, J. Amer. Math. Soc. {\bf 28} (2015), no.~1, 235--278. 
\bibitem{dF1}de Fernex, T.: {\it Birationally rigid hypersurfaces}, Invent. Math. {\bf 192} (2013), no.~3, 533--566. 
\bibitem{dF2}de Fernex, T.: {\it Erratum to: Birationally rigid hypersurfaces}, Invent. Math. {\bf 203} (2016), no.~2, 675--680.(2016),
\bibitem{dFEM1}de Fernex, T., Ein, L., Musta\c t\u a, M.: {\it Multiplicities and log canonical threshold}, J. Algebraic Geom. {\bf 13} (2004), no.~3, 603--615. 
\bibitem{dFEM2}de Fernex, T., Ein, L., Musta\c t\u a, M.: {\it Bounds for log canonical thresholds with applications to birational rigidity}, Math. Res. Lett. {\bf 10} (2003), no.~2-3, 219--236. 
\bibitem{EM}Ein, L., Musta\c t\v a, M.: {\it Inversion of adjunction for local complete intersection varieties}, Amer. J. Math. {\bf 126} (2004), no.~6, 1355--1365. 
\bibitem{Fuj}Fujita, K.: {\it K-stability of Fano manifolds with not small alpha invariants}, 2016, To appear in J. Inst. Math. Jussieu, arXiv:1602.01305.
\bibitem{F}Fulton, W.: {\it Intersection theory}, second edition, Ergebnisse der Mathematik und ihrer Grenzgebiete. 3. Folge. A Series of Modern Surveys in Mathematics, 2, Springer, Berlin, 1998.
\bibitem{G}Grothendieck, A.: {\it Cohomologie locale des faisceaux coh\'erents et th\'eor\`emes de Lefschetz locaux et globaux $(SGA$ $2)$}, North-Holland, Amsterdam, 1968.
\bibitem{IM}Iskovskih, V. A., Manin, Ju. I.: {\it Three-dimensional quartics and counterexamples to the L\"uroth problem}, Mat. Sb. (N.S.) {\bf 86(128)} (1971), 140--166.
\bibitem{KOW}Kim, I. -K., Okada, T., Won, J.: {\it Alpha invariants of birationally rigid Fano three-folds}, Int. Math. Res. Not. IMRN {\bf 2018}, no.~9, 2745--2800.
\bibitem{K1}Koll\'ar, J.: {\it Singularities of pairs}, in {\it Algebraic geometry---Santa Cruz 1995}, 221--287, Proc. Sympos. Pure Math., 62, Part 1, Amer. Math. Soc., Providence, RI. 
\bibitem{K2}Koll\'ar, J.: {\it The rigidity theorem of Fano--Segre--Iskovskikh--Manin--Corti--Pukhlikov--Cheltsov--de Fernex--Ein--Musta\c t\u a--Zhuang}, 2018, arXiv:1807.00863.
\bibitem{L}Lazarsfeld, R.: {\it Positivity in algebraic geometry. I}, results of mathematics and its border areas. 3rd episode. A Series of Modern Surveys in Mathematics, 48, Springer-Verlag, Berlin, 2004.
\bibitem{LZ}Liu, Y., Zhuang, Z.: {\it Birational superrigidity and K-stability of singular Fano complete intersections}, 2018, to appear in Int. Math. Res. Not. IMRN, arXiv:1803.08871.
\bibitem{OO}Odaka, Y., Okada, T.: {\it Birational superrigidity and slope stability of Fano manifolds}, Math. Z. {\bf 275} (2013), no.~3-4, 1109--1119.
\bibitem{P1}Pukhlikov, A. V.: {\it Birational isomorphisms of four-dimensional quintics}, Invent. Math. 87 (1987), no. 2, 303--329.
\bibitem{P2}Pukhlikov, A. V.: {\it Birational automorphisms of Fano hypersurfaces}, Invent. Math. {\bf 134} (1998), no. 2, 401-426.
\bibitem{P3}Pukhlikov, A. V.: {\it Birationally rigid Fano complete intersections}, Crelle J. f\" ur die reine und angew. Math. {\bf 541} (2001), 55-79.
\bibitem{P4}Pukhlikov, A. V.: {\it Birationally rigid Fano hypersurfaces}, Izv. Ross. Akad. Nauk Ser. Mat. {\bf 66} (2002), no. 6, 159--186; translation in Izv. Math. {\bf 66} (2002), no.~6, 1243--1269. 
\bibitem{P5}Pukhlikov, A. V.: {\it Birationally rigid complete intersections of quadrics and cubics}, Izv. Ross. Akad. Nauk Ser. Mat. {\bf 77} (2013), no. 4, 161--214; translation in Izv. Math. {\bf 77} (2013), no.~4, 795--845. 
\bibitem{P6}Pukhlikov, A. V.: {\it Birationally rigid Fano complete intersections. II}, J. Reine Angew. Math. {\bf 688} (2014), 209--218.
\bibitem{SZ}Stibitz, C., Zhuang, Z.: {\it K-stability of birationally superrigid Fano varieties}, Compos. Math. {\bf 155} (2019), no.~9, 1845--1852. 
\bibitem{S}Suzuki, F.: {\it Birational rigidity of complete intersections}, Math. Z. {\bf 285} (2017), no.~1-2, 479--492. 
\bibitem{T}Tian, G.: {\it K-stability and K\"{a}hler-Einstein metrics}, Comm. Pure Appl. Math. {\bf 68} (2015), no.~7, 1085--1156.
\bibitem{Za}Zak, F. L.: {\it Tangents and secants of algebraic varieties}, translated from the Russian manuscript by the author, Translations of Mathematical Monographs, 127, American Mathematical Society, Providence, RI, 1993
\bibitem{Z}Zhuang, Z.: {\it Birational superrigidity and K-stability of Fano complete intersections of index one} (with an appendix written jointly with Charlie Stibitz), 2018. arXiv:1802.08389.
\end{thebibliography}
\end{document}